\documentclass[a4paper,12pt]{article}
\usepackage{amsthm}
\usepackage{amsfonts}
\usepackage{amsmath}
\usepackage{amssymb}
\usepackage{multicol}
\usepackage{mathrsfs}
\usepackage{thmtools}
\usepackage{thm-restate}

\title{\textbf{Size reconstructibility of graphs}}
\author{Carla Groenland\thanks{Mathematical Institute, University of Oxford, Oxford OX2 6GG, United Kingdom. \texttt{\{groenland,guggiari,scott\}@maths.ox.ac.uk}}, Hannah Guggiari\footnotemark[1], Alex Scott\footnotemark[1] \thanks{Supported by a Leverhulme Trust Research Fellowship.}\\
}
\date{\today}

\newtheorem{theorem}{Theorem}[section]
\newtheorem{conjecture}[theorem]{Conjecture}
\newtheorem{lemma}[theorem]{Lemma}
\newtheorem{observation}[theorem]{Observation}
\newtheorem{corollary}[theorem]{Corollary}
\newtheorem{problem}[theorem]{Problem}

\newcommand{\missing}{\frac1{20}}
\newcommand{\adversary}{\frac1{40}}

\newtheorem{claim}{Claim}
\makeatletter
\newenvironment{claimproof}[1][\proofname]{\par\pushQED{\hfill$\lozenge$}\normalfont \topsep6\p@\@plus6\p@\relax \trivlist \item\relax{\itshape#1\@addpunct{.}}\hspace\labelsep\ignorespaces}{\popQED\endtrivlist\@endpefalse}
\makeatother

\begin{document}
\newpage
\maketitle
\begin{abstract}
\noindent
The deck of a graph $G$ is given by the multiset of (unlabelled) subgraphs $\{G-v:v\in V(G)\}$. The subgraphs $G-v$ are referred to as the cards of $G$. Brown and Fenner recently showed that, for $n\geq29$, the number of edges of a graph $G$ can be computed from any deck missing 2 cards. We show that, for sufficiently large $n$, the number of edges can be computed from any deck missing at most $\missing\sqrt{n}$ cards.
\end{abstract}

\section{Introduction}
\noindent
Throughout this paper, all graphs are finite and undirected with no loops or multiple edges. The \textit{order} of a graph is the number of vertices in the graph; the \textit{size} of a graph refers to the number of edges.

Given a graph $G$ and any vertex $v\in V(G)$, the \textit{card} $G-v$ is the subgraph of $G$ obtained by removing the vertex $v$ and all edges incident to $v$. The multiset $\mathcal{D}(G)$ of all unlabelled cards of $G$ is called the \textit{deck} and has size $n$.

It is natural to ask whether it is possible for two non-isomorphic graphs to have the same deck. Kelly and Ulam \cite{K42,K57,U60} proposed the following \textit{Reconstruction Conjecture}.
\begin{conjecture}
For $n>2$, two graphs $G$ and $H$ of order $n$ are isomorphic if and only if $\mathcal{D}(G)=\mathcal{D}(H)$.
\end{conjecture}
\noindent
The Reconstruction Conjecture remains open, although it is known to be true for a few classes of graphs (for example, trees \cite{K57}). Moreover, almost every graph can be reconstructed \cite{B90,M76,M88}. For more background, see \cite{AFLM10, B91, BH77, LS16, N78}.

A more general problem is to determine which parameters of a graph can be calculated from its deck. Such parameters are said to be \textit{reconstructible}. Given a full deck of cards, it is easy to reconstruct the number of edges $m$: summing over the edges present in all of the cards gives $m(n-2)$, where $n$ is the number of vertices. It is also well known that connectedness and the degree sequence are reconstructible.

Some parameters are reconstructible even if there is not a full deck of cards. For example, Bowler, Brown, Fenner and Myrvold \cite{BBFM11} showed that any $\left\lfloor\frac n2\right\rfloor+2$ cards suffice to determine whether the graph is connected. Myrvold \cite{M92} also found that the degree sequence is reconstructible from any $n-1$ cards.

In this paper, we are concerned with reconstructing the number of edges. Myrvold's result \cite{M92} on the degree sequence immediately implies that the size is reconstructible from any $n-1$ cards. In a recent paper, Brown and Fenner \cite{BF17} showed that, for $n\geq29$, the size of a graph can be reconstructed from any $n-2$ cards. 

Woodall \cite{W15} found that, for any $p\geq3$ and $n$ sufficiently large, if two graphs on $n$ vertices have $n-p$ common cards, then the number of edges in these two graphs differs by at most $p-2$.

In Section \ref{sec:main}, we will improve on both results by showing that the size of a graph is reconstructible with up to $\missing\sqrt{n}$ missing cards. In particular, we will prove the following theorem.

\begin{restatable}{theorem}{sizereconstructability}
\label{thm:sizereconstructibility}
For $n$ sufficiently large and $k\leq \missing\sqrt{n}$, the number of edges $m$ of a graph $G$ on $n$ vertices is reconstructible from any $n-k$ cards.
\end{restatable}

We will also consider the following adversarial version of the problem. An adversary chooses a graph $G$ of order $n$ and gives us a collection of $n$ cards, each showing a graph on $n-1$ vertices. We are told that there are $n-k$ \textit{true cards}, which come from the deck $\mathcal{D}(G)$. The other $k$ cards are \textit{false cards}, which can depict any graph of order $n-1$. For which $k$ can we reconstruct the size of $G$, regardless of the graph $G$ and the cards given by the adversary? Theorem \ref{thm:sizereconstructibility} immediately implies the following.

\begin{corollary}
\label{cor:adversary}
Let $n$ be sufficiently large and $k\leq \adversary\sqrt{n}$. The number of edges $m$ of a graph $G$ on $n$ vertices is reconstructible from any collection  $\mathcal{C}$ of cards where $n-k$ are true and $k$ are false.
\end{corollary}
\begin{proof}
Suppose that $G$ and $H$ are two graphs on $n$ vertices and each has at least $n-k$ cards in common with a deck of cards $\mathcal{C}$. Then $G$ and $H$ must have at least $n-2k$ cards in common. We may apply Theorem \ref{thm:sizereconstructibility} to these $n-2k$ common cards. If $n$ is sufficiently large and $2k\leq\missing\sqrt{n}$, then $G$ and $H$ must have the same number of edges.
\end{proof}
The rest of the paper is organised as follows. Theorem \ref{thm:sizereconstructibility} is proved in Section \ref{sec:main} and some open problems are given in Section \ref{sec:conclusion}. 

\section{Size reconstruction from $n-k$ cards}
\label{sec:main}
We first give the relevant definitions in Section \ref{subsec:notation} followed by an outline of our proof in Section \ref{subsec:proof_overview}. Some of the auxiliary results are given in Section \ref{subsec:preliminary_results} and the main proof is presented in Section \ref{subsec:main_result}.

\subsection{Notation and definitions}
\label{subsec:notation}
Throughout Section \ref{sec:main}, $G$ is a graph of order $n$ and size $m=e(G)$, where $m$ is unknown. The vertex set of $G$ is $V(G)=\{v_1,\dots,v_n\}$ and we write $G_i$ for the card $G-v_i$. We may assume that we are given the cards $G_1,\dots,G_{n-k}$. In the proof of the main result, we will assume that $k\leq \missing\sqrt{n}$. 

For any graph $H$, let the number of vertices of degree $t$ be 
\begin{equation*}
d_t(H)=|\{v\in V(H):d_{H}(v)=t\}|
\end{equation*}
where $d_{H}(v)$ denotes the degree of $v$ in $H$. For convenience, we write $d_t=d_t(G)$ and $d(v)=d_G(v)$. Note that $d_t$ is unknown for every $t$ and that we know $d_t(G_1),\dots,d_t(G_{n-k})$.

Let $s_t = \sum_{i=1}^n d_t(G_i)$. As we will note below (Lemma \ref{lem:magicformula}), it is easy to see that
\begin{equation}
\label{eq:magicformula_notation}
s_t= \sum_{i=1}^n d_t(G_i) = (n-1-t)d_t + (t+1)d_{t+1}.
\end{equation}

As we progress in the proof, we will use various estimates determined from the cards for quantities of interest. We set
\begin{equation*}
\widetilde{m}=\left\lfloor\frac1{n-2-k}\sum_{i=1}^{n-k}e(G_i)\right\rfloor
\end{equation*}
\noindent as an estimate of the number of edges $m$,
\begin{equation*}
\widetilde{d}_t=|\{i\in\{1,\dots,n-k\}:\widetilde{m} - e(G_i) =t\}|
\end{equation*}
\noindent as an estimate of the number $d_t$ of vertices of degree $t$, and
\begin{equation*}
\widetilde{s}_t=\sum_{i=1}^{n-k} d_t(G_i)
\end{equation*}
as an estimate of $s_t = \sum_{i=1}^n d_t(G_i)$ (thus $s_t$ is the number of degree $t$ vertices in the full deck of cards, while $\widetilde{s_t}$ is the number of degree $t$ vertices on the cards that we are allowed to see).

We use the short-hand $[n]=\{1,\dots,n\}$ and slightly abuse notation by writing $[a,b]=[a,b]\cap \mathbb{Z}$ for the set of integers in the corresponding real interval.
\subsection{Proof overview}
\label{subsec:proof_overview}
We first show that our estimate $\widetilde{m}$ on the number of edges $m$ is an upper bound on $m$ satisfying $0\leq \widetilde{m}-m<2k$. Our goal is then to determine $\alpha=\widetilde{m}-m$ from the cards, since this allows us to compute $m$ from $\widetilde{m}$. 

If we knew the number of edges $m$, then we could calculate the degree of vertex $v_i$ from its card $G_i$ by setting $d(v_i)=m-e(G_i)$. Instead, we estimate the degree of the vertex corresponding to each card by 
\begin{equation*}\widetilde{d}(v_i)=\widetilde{m}-e(G_i)\end{equation*} and count the number of vertices with estimated degree $t$
\begin{equation*}
\widetilde{d}_t=|\{i\in[n-k]:\widetilde{m}-e(G_i)=t\}|.
\end{equation*}
Since $m\leq\widetilde{m}$, our estimate $\widetilde{d}(v_i)$ may be larger than the actual degree of vertex $v_i$. This means that the actual sequence $(d_t)$ has been shifted to the right by $\alpha$. Moreover, $k$ degrees did not get counted due to the missing cards. It is important to notice here that we know the shift is equal to $\alpha$, even when we might not know any of the $d_t$ or $\alpha$ itself. 

We note that $d_t-k\leq \widetilde{d}_{t+\alpha}\leq d_t$. Hence, if we were told that $d_t>k$ and $d_{t+1}=\dots=d_{t+2k}=0$, then we could determine the shift $\alpha$ from $(\widetilde{d}_t)$ (namely, $\alpha$ would be the largest $i\in \{0,\dots,2k\}$ for which $\widetilde{d}_{t+i}>0$). Aiming for a situation like this, we reconstruct $d_t$ exactly from the cards for many values of $t$. If we know $d_{t+1}$, then the formula given in \eqref{eq:magicformula_notation} makes it possible to compute $d_t$ from $s_t$. Unfortunately, we cannot determine $s_t$ exactly but an estimate $\widetilde{s}_t$ suffices in many cases: if we can compute an estimate for the integer $d_t$ with error less than $\frac12$, then we can round away the error. 
This is made precise in Claim 1. 

In Lemma \ref{lem:findingdt}, we show that, for many values of $t$, we can ``guess'' the integers $d_t$ and $d_{t+1}$ from $\widetilde{s}_t$. We require the value $\frac{t+1}n$ to be bounded away from certain fractions (that do not depend on $G$). Moreover, we need $d_t$ and $d_{t+1}$ to be small (to improve the estimate $\widetilde{s}_t$ and to have fewer values to guess between). In order to find a $t$ for which $d_t$ and $d_{t+1}$ are small, we compute yet another estimate $d_t^*$ from the cards in Lemma \ref{lem:dtstar}.

Using our reconstructed values for $d_t$, we reconstruct the shift $\alpha=\widetilde{m}-m$ which allows us to determine $m$.

\subsection{Preliminary results}
\label{subsec:preliminary_results}
As noted above, we set
\begin{equation*}
\widetilde{m}=\left\lfloor\frac1{n-2-k}\sum_{i=1}^{n-k}e(G_i)\right\rfloor.
\end{equation*}
We will use $\widetilde{m}$ as an estimate for the number of edges in $G$. Let 
\begin{equation*}
\alpha = \widetilde{m}-m.
\end{equation*}
We can calculate $\widetilde{m}$ from the cards $G_1,\dots,G_{n-k}$. Thus in order to determine $m$, it is enough to determine the ``shift'' $\alpha$.
\begin{lemma}
\label{lem:edges}
$0\leq \widetilde{m}-m \leq \frac{k(n-1)}{n-2-k}$.
\end{lemma}
\noindent
Note that, if $k=o(n)$, then $\alpha = \widetilde{m}-m \leq(1+o(1))k$.
\begin{proof}[Proof of Lemma \ref{lem:edges}]
Suppose that we have the entire deck of $G$. Every edge of $G$ is on exactly $n-2$ cards and therefore $\sum_{i=1}^n e(G_i)=(n-2)m$. Furthermore, for every $v_i\in V(G)$, we have that $e(G_i)=m-d(v_i)$. It follows that 
\begin{align*}
\sum_{i=1}^{n-k}e(G_i)&=(n-2)m -\sum_{i=n-k+1}^n e(G_i)\\
&=(n-2-k)m+\sum_{j=n-k+1}^n d(v_{j}).
\end{align*}
The claimed bounds follows from the fact that $0\leq d(v)\leq n-1$ for all $v\in V(G)$.
\end{proof}
\noindent
For $t\in \{0,\dots, n-1\}$, recall that $s_t=\sum_{i=1}^n d_t(G_i)$ and
\begin{equation*}
\widetilde{s}_t  = \sum_{i=1}^{n-k} d_t(G_i) = \sum_{i=1}^{n-k} |\{v\in V(G_i):d_{G_i}(v)=t\}|.
\end{equation*}
Note that $\widetilde{s}_t$ can be calculated from the given cards.
\begin{lemma}
\label{lem:magicformula}
We have $d_t(G_i)\leq d_t+d_{t+1}$ and
\begin{equation}
\label{eq:magicformula}
s_t = \sum_{i=1}^{n} d_t(G_i)=(n-1-t)d_t+(t+1)d_{t+1}.
\end{equation}
In particular, $0\leq  s_t-\widetilde{s}_t \leq k(d_t+d_{t+1})$.
\end{lemma}
\begin{proof}
A vertex of degree $t$ on a card $G_i$ can either have degree $t$ in the graph $G$ or degree $t+1$ (in the case where it is a neighbour of $v_i$). This shows that $d_t(G_i)\leq d_t+d_{t+1}$ for all $i$. 

A vertex of degree $t+1$ gets counted exactly once in $\sum_{i=1}^nd_t(G_i)$ for each of its neighbours; a vertex of degree $t$ gets counted on all cards except for its own and those of its neighbours. This proves \eqref{eq:magicformula}. The last claim follows by combining the fact that $s_t-\widetilde{s}_t= \sum_{j=n-k+1}^nd_t(G_{j})$ with the first claim.
\end{proof}
\noindent
As noted by Brown and Fenner \cite{BF17} and others, any result for a graph $G$ implies a corresponding result for its complement $\overline{G}$.
\begin{observation}
\label{obs:complementarity}
If $\mathcal{D}(G)=\{G_1,\dots,G_n\}$, then $\mathcal{D}(\overline{G})=\{\overline{G}_1,\dots,\overline{G}_n\}$. Moreover, we have that $d_t(\overline{G})=d_{n-1-t}(G)$ for any $t\in\{0,\dots,n-1\}$.
\end{observation}
\noindent
The result below will be used to find values of $t$ for which $d_t$ is guaranteed to be small. 
\begin{lemma}
\label{lem:dtstar}
Suppose that $k\leq\frac n3$. For each $t\in\{0,\dots,n-1\}$ we can calculate a value $d_t^*$ from the cards that satisfies $\frac14d_t-1\leq d_t^*\leq d_{t-1}+d_t+d_{t+1}$.
\end{lemma}
\begin{proof}
We will consider two cases: when $t<\frac n2$ and when $t\geq\frac n2$.\\
\newline
\textit{Case 1:} $t<\frac n2$.\\
Define
\begin{equation}
\label{eq:dtstar1}
d_t^*=d_t^*(G)=\max\{d_t(G_i):1\leq i\leq n-k\}.
\end{equation}
Note that $d_t^*$ can be calculated from the given cards and that $d_t^*\leq d_t+d_{t+1}$ by Lemma \ref{lem:magicformula}.

Let $N$ be the number of times a vertex of degree $t$ in $G$ is seen as a vertex of degree $t-1$ in the cards $G_1,\dots,G_{n-k}$. We will find upper and lower bounds for $N$. For the upper bound, note that a vertex of degree $t$ appears as a vertex of degree $t-1$ on the card $G_i=G-v_i$ if and only if $v_i$ is one of its neighbours. Therefore, $N\leq td_t$.

Now consider the card $G_i$ for some $i\in[n-k]$. We claim that there are at least $d_t-1-d_t(G_i)$ vertices that have degree $t-1$ in $G_i$ but degree $t$ in $G$. Indeed, the only missing vertex is $v_i$ (which might have degree $t$) and at most $d_t(G_i)$ of the other vertices with degree $t$ in $G$ have degree $t$ in $G_i$. It follows that $N\geq\sum_{i=1}^{n-k}(d_t-1-d_t(G_i))$. We combine these bounds on $N$ to get
\begin{equation*}
td_t\geq N\geq\sum_{i=1}^{n-k}(d_t-1-d_t(G_i))\geq(n-k)(d_t-d_t^*-1).
\end{equation*}
Rearranging and using the assumptions that $t<\frac n2$ and $n-k\geq\frac{2n}3$, we find $\frac23d_t^*\geq\frac16d_t-\frac23$. It follows that $d_t^*\geq\frac14d_t-1$.\\
\newline
\textit{Case 2:} $t\geq\frac n2$.\\
Define
\begin{equation}
\label{eq:dtstar2}
d_t^*=d_{n-1-t}^*(\overline{G}).
\end{equation}
As $n-1-t<\frac n2$, this is well-defined. From the argument above, we have
\begin{equation*}
\frac14d_{n-1-t}(\overline{G})-1\leq d_{n-1-t}^*(\overline{G})\leq d_{n-1-t}(\overline{G})+d_{n-t}(\overline{G}).
\end{equation*}
By Observation \ref{obs:complementarity}, we see that
\begin{equation*}
\frac14d_t(G)-1\leq d_{n-1-t}^*(\overline{G})=d_t^*\leq d_t(G)+d_{t-1}(G).
\end{equation*}
As $d_{t-1}$ and $d_{t+1}$ are both non-negative for every value of $t$, the result follows.
\end{proof}
\noindent
In the proof of Theorem \ref{thm:sizereconstructibility}, we will compare the unknown sequence $(d_t)$ to a sequence $(\widetilde{d}_t)$ that can be calculated from the cards. In order to do this, we will need to know some values of $d_t$ exactly.
For the proof we will only need the following lemma in the case when $\beta=\frac12$ and $t$ lies in the interval $[\frac n3,\frac{2n}3]$. However, the result may be useful elsewhere and so we state it in a more general form. 
\begin{lemma}
\label{lem:findingdt}
Suppose $0\leq\beta<1$ and let $\gamma=\frac34+\frac14\beta$. Suppose $n$ is sufficiently large and $k=O(n^{\beta})$. Then, for any graph $G$ of order $n$ and any deck of $n-k$ cards, the value of $d_t$ can be calculated exactly for all but $O(n^{\gamma})$ values of $t$.
\end{lemma}
\begin{proof}
Recall from Lemma \ref{lem:magicformula} that
\begin{equation*}
s_t = \sum_{i=1}^{n} d_t(G_i)=(n-1-t)d_t+(t+1)d_{t+1}
\end{equation*}
and that $\widetilde{s}_t=\sum_{i=1}^{n-k}d_t(G_i)$ approximates $s_t$ where $0\leq s_t-\widetilde{s}_t \leq k(d_t+d_{t+1})$. Let $q=\frac{t+1}n\in[0,1]$. Then $\frac{s_t}n=(1-q)d_t+qd_{t+1}$ and 
\begin{equation*}
\left|\frac{s_t}n-\frac{\widetilde{s_t}}n\right| \leq \frac{k(d_t+d_{t+1})}n.
\end{equation*}
Our goal will be to find values of $t$ for there is only one choice of $(a,b)$ such that $\left|(1-q)a+qb-\frac{\widetilde{s}_t}n\right|\in\left[0,\frac{k(d_t+d_{t+1})}n\right]$. 

To achieve this, we first restrict to those values of $t$ for which we can calculate an upper bound on $d_t$ and $d_{t+1}$ from the cards. Assume that $n$ is sufficiently large to ensure $k\leq\frac n3$. Lemma \ref{lem:dtstar} then applies to ensure that, for all $t$ the quantity $d_t^*$ (which is defined in \eqref{eq:dtstar1} and \eqref{eq:dtstar2} and can be calculated from the cards) satisfies $\frac14d_t-1\leq d_t^*\leq d_{t-1}+d_t+d_{t+1}$. By the lower bound, if $d_t^*$ is small then $d_t$ is small as well. We use the upper bound to show that $d_t^*$ is small for most values of $t$.
Indeed, let $K=n^{1-\gamma}$, $I=\{0,\dots,n-1\}$ and $A=\{t\in I:d_t^*+1\geq\frac14K\}$. Then
\begin{equation*}
\frac{1}4K|A|\leq \sum_{t\in A}(d_t^*+1)\leq\sum_{t\in A} (d_{t-1}+d_t+d_{t+1}+1)\leq4n.
\end{equation*}
and hence $|A|\leq 16n/K=16n^\gamma$. For all $t$ in the set $I'=\{t\in I:t,t+1\not\in A\}$, we know that $d_t,d_{t+1}< K$. Since $|A|=|\{a:a+1\in A\}|$, by restricting to $I'$, we remove at most $O(n^\gamma)$ potential $t$. 

For all $t\in I'$, we know that 
\begin{equation*}
0\leq (1-q)a+qb-\frac{\widetilde{s}_t}n \leq \frac{k(d_t+d_{t+1})}n< \frac{2Kk}n
\end{equation*}
It remains to determine for which $q=\frac{t+1}n$ the following holds: any two elements in $X=\left\{(1-q)a+qb:a,b\in\left\{0,\dots\left\lfloor K\right\rfloor\right\}\right\}$ take values that are at least $\frac{4Kk}n$ apart, so that there is at most one $(1-q)a+qb\in X$ within $\frac{2Kk}n$ of $\widetilde{s}_t$. For all such $t\in I'$, we can then reconstruct $d_t$ and $d_{t+1}$ from the cards as the unique choices for $a$ and $b$.

Let $M=\frac{4Kk}n$. Suppose that, for some $\delta< M$, we are able to find elements $a>a'$ and $b<b'$ within $\left\{0,\dots,\left\lfloor K\right\rfloor\right\}$ satisfying $a(1-q)+bq=a'(1-q)+b'q+\delta$. Rearranging, we get
\begin{equation*}
a-a'=(b'-b+a-a')q+\delta.
\end{equation*}
In particular, $(b'-b+a-a')q+\delta$ is an integer. As $b'-b+a-a'\in \left\{1,\dots,\left\lfloor2K\right\rfloor\right\}$, it suffices to ensure that, for all $y\in \left\{1,\dots,\left\lfloor2K\right\rfloor\right\}$, $yq$ is at distance at least $M$ from all integers $x\in\left\{1,\dots,\left\lfloor K\right\rfloor\right\}$. Let
\begin{equation*}
R=\left\{\frac xy:x\in\left\{1,\dots,\left\lfloor K\right\rfloor\right\},y\in\left\{1,\dots,\left\lfloor2K\right\rfloor\right\}\right\}
\end{equation*}
and
\begin{equation*}
S=\left\{t:\exists r\in R\text{ such that }\left|\frac{t+1}n-r\right|<M\right\}.
\end{equation*}
As argued above, for each $t\in I'\setminus S$ we are able to ``guess'' the values of $d_t$ and $d_{t+1}$. It remains to bound the size of $S$. The set $R$ has size less than $2K^2$. For each choice of $r\in R$, there are at most $2Mn$ elements of the form $\frac{i}n$ with $i\in \{0,\dots,n-1\}$ that are within $M$ of $r$. This shows that $|S|\leq 2Mn|R|\leq 16kK^3$. Recall that $k= O(n^\beta)$,  $16K^3=O(n^{3(1-\gamma)})$ and $\gamma=\frac34+\frac14\beta$. We calculate
\begin{equation*}
\beta+3(1-\gamma)=\beta+3\left(\frac14-\frac1{4}\beta\right)=\gamma.
\end{equation*}
Let $J=I'\setminus S$. For every $t\in J$, we can calculate $d_t$ exactly and furthermore $|I\setminus J|=|(I\setminus I')\cup S)|=O(n^\gamma)$ as desired.
\end{proof}
\noindent
Since $\gamma<1$, the result shows that we can reconstruct $d_t$ for all but $o(n)$ of the $t\in [0,n]$.

\subsection{Proof of main result}
\label{subsec:main_result}
We are now ready to prove Theorem \ref{thm:sizereconstructibility}, which is restated below.
\sizereconstructability*

\begin{proof}
Let $n$ be sufficiently large and $k=\lfloor \frac1{20}\sqrt{n}\rfloor$. Let $G$ be a graph on $n$ vertices and let $G_1,\dots,G_{n-k}$ be the $n-k$ cards of $G$ that we are given.

Our goal is to determine $d_t$ for many values of $t$. We will handle values of $t$ for which $d_t>\sqrt{n}$ separately from those $t$ where $d_t\leq \sqrt{n}$.  For this reason, it will be convenient to say that $d_t$ is \textit{big} if $d_t>\sqrt{n}$ and \textit{little} if $d_t\leq \frac34\sqrt{n}$.

\begin{claim}
Suppose that, for some $t\leq\frac{2n}3-1$, the value of $d_{t+1}$ is known exactly and is not big. Then either $d_t$ can be calculated exactly or $d_t$ can be identified as being big.
\end{claim}
\begin{claimproof}
Since we can calculate $\widetilde{s}_t=\sum_{i=n-k+1}^n d_t(G_i)$ from the cards, if $d_{t+1}$ is known, then we can calculate 
\begin{equation*}
d'_t=\frac1{n-1-t}(\widetilde{s}_t-(t+1)d_{t+1})
\end{equation*}
from the cards. By Lemma \ref{lem:magicformula}, 
\begin{equation*}
d_t = d_t' +\frac{s_t-\widetilde{s}_t}{n-1-t}
\end{equation*}
where $0\leq s_t-\widetilde{s}_t\leq k(d_t+d_{t+1})$. In particular $d_t\geq d'_t$, so we recognise that $d_t$ is big if $d'_t> \sqrt{n}$. We now show that, if $d'_t\leq \sqrt{n}$, then the closest integer to $d'_t$ equals $d_t$.

Since $t+1\leq \frac{2n}3$ and $d_{t+1}$ is not big,
\begin{equation}
\label{eq:calculation}
\frac{s_t-\widetilde{s}_t}{n-1-t}\leq \frac3n k(d_t+d_{t+1})\leq \frac3{n} k(d_t+\sqrt{n})\leq \frac3{20\sqrt{n}}(d_t+\sqrt{n}).
\end{equation}
We conclude that $d_t-d_t'<\frac12$ if $d_t\leq 2\sqrt{n}$. Hence the closest integer to $d'_t$ equals $d_t$ in this case. 

From the calculation in \eqref{eq:calculation} we also find
\begin{equation*}
d_t'\geq d_t-\frac{s_t-\widetilde{s}_t}{n-1-t} > d_t - \frac3{20\sqrt{n}}(d_t+\sqrt{n}) \geq \frac12 d_t> \sqrt{n}
\end{equation*}
if $d_t>2\sqrt{n}$. Hence either $d'_t>\sqrt{n}$ (in which case $d_t$ is big) or rounding it to the nearest integer gives us $d_t$ exactly.
\end{claimproof}

\begin{claim}
Suppose that, for some $t\geq\frac n3+1$, the value of $d_{t-1}$ is known exactly and is not big. Then either $d_t$ can be calculated exactly or $d_t$ can be identified as being big.
\end{claim}
\begin{claimproof}
If $t\geq \frac n3+1$, then $n-t-1\leq\frac{2n}3-1$. By Observation \ref{obs:complementarity}, we have $d_{n-t}(\overline{G})=d_{t-1}(G)$. Apply Claim 1 to $\overline{G}$ to see that either $d_t(G)=d_{n-t-1}(\overline{G})$ can be calculated exactly or it can be identified as being big.
\end{claimproof}
\begin{claim}
The interval $[\frac n3,\frac{2n}3]$ contains $2k$ consecutive values of $t$ such that every $d_t$ can be calculated exactly and they are all little.
\end{claim}
\begin{claimproof}
Let $I=[\frac n3,\frac{2n}3]\cap\mathbb{N}$. Lemma \ref{lem:findingdt} with $\beta=\frac12$ gives a set $J\subseteq I$ and a constant $c$ such that $|J|\leq cn^{\frac78}$ and we can calculate $d_t$ exactly if $t\in I\setminus J$.

Partition $I$ into $\left\lfloor\frac n{6k}\right\rfloor$ intervals of length $2k$. At most $\left\lfloor\frac{cn^{7/8}}{2k}\right\rfloor$ of them are completely contained in $J$. For $n$ sufficiently large, $\left\lfloor\frac n{6k}\right\rfloor-\left\lfloor\frac{cn^{7/8}}{2k}\right\rfloor\geq\frac n{8k}$. Therefore, for these values of $n$, there are at least $\frac n{8k}$ intervals which are not completely contained within $J$. By Claims 1 and 2, we are able to calculate $d_t$ exactly for all values of $t$ in each of these intervals unless the interval happens to contain a value of $t$ for which $d_t$ is big.

We know that there are at most $\frac43\sqrt{n}$ values of $t\in\{0,\dots,n-1\}$ for which $d_t$ is not little. Therefore, as $\frac n{8k}\geq \frac{5}2\sqrt{n}>\frac43\sqrt{n}$, there exists an interval which is not completely contained within $J$ and which only contains values of $d_t$ that are little, each of which we can calculate exactly.
\end{claimproof}
\noindent
By Claim 3, we can find an interval $\mathcal{I}=\{b,b+1,\dots,b+2k-1\}\subset[\frac n3,\frac{2n}3]$ such that, for every $t\in\mathcal{I}$, we can calculate $d_t$ exactly and it is little. We may then recursively apply Claim 1, starting with $t+1=b$. We continue until either we reach $d_0$ or we hit a big vertex $d_{t_\ell}$ for some $t_\ell<b$. Similarly, we may recursively apply Claim 2, starting with $t-1=b+2k-1$. Again, we will either calculate $d_{n-1}$ or we will identify that $d_{t_r}$ is big for some $t_r>b+2k-1$.

If we are able to calculate both $d_0$ and $d_{n-1}$, then we will know $d_t$ for every $t\in\{0,\dots,n-1\}$. This tells us the degree sequence of $G$ and hence we can directly calculate $m$.

Therefore, we may assume that we have the following situation: there exists an interval $\mathcal{J}\supseteq\mathcal{I}$ with endpoints $t_\ell$ and $t_r$ such that $t_\ell<t_r$. For every $t\in\mathcal{J}\setminus\{t_\ell,t_r\}$, the value $d_t$ is known exactly and is not big. At least one of $d_{t_\ell}$ and $d_{t_r}$ has been identified as being big. By Observation \ref{obs:complementarity}, we may assume that $d_{t_\ell}$ is big.

By Lemma \ref{lem:edges}, the estimate $\widetilde{m}$ for $m$ that we can obtain from the cards $G_1,\dots,G_{n-k}$ satisfies $\widetilde{m}=m+\alpha$ with $0\leq\alpha\leq \left\lfloor\frac{k(n-1)}{n-2-k}\right\rfloor$. For $n$ sufficiently large, we have $n-1<2(n-2-k)$ and hence $\alpha<2k$. Recall from the proof overview that $\widetilde{d}_t=|\{i\in\{1,\dots,n-k\}:\widetilde{m}-e(G_i)=t\}|$ can be calculated from the cards and that our goal is to discover the ``shift'' $\alpha=\widetilde{m}-m$ in this sequence. The overall shape of $\widetilde{d}_0,\dots,\widetilde{d}_{n-1}$ will be the same as the overall of shape of $d_0,\dots,d_{n-1}$ but shifted to the right by $\alpha$. Moreover, we are ``missing'' $k$ values, so that $\sum_{t=0}^{n-1}|d_t-\widetilde{d}_{t+\alpha}|=k$. (Note that we need to calculate $\widetilde{d}_t$ for $0\leq t\leq n+2k$ and that, for $t+\alpha\geq n$, it is possible for $\widetilde{d}_{t+\alpha}$ to take a non-zero value.) 

Although we do not know the exact value of $d_{t_\ell}$, it is sufficient to redefine each $d_{t}$ and $\widetilde{d}_t$ to be the minimum of their current value and $\sqrt{n}$. After doing this, we still have $\sum_{t=0}^{n-1}|d_t-\widetilde{d}_{t+\alpha}|\leq k$. It follows that $\sum_{t=t_\ell}^{t_r-1}|d_t-\widetilde{d}_{t+\alpha}|\leq k$. We now show that $\alpha$ can be recognised as the unique ``shift'' $s$ in a given interval that ensures $\widetilde{d}_{t+s}$ is sufficiently close to $d_t$.

\begin{claim}
For $s\in\{0,\dots,2k-1\}$, $\sum_{t=t_\ell}^{t_r-1}|d_t-\widetilde{d}_{t+s}|\leq k$ if and only if $s=\alpha$.
\end{claim}
\begin{claimproof}
Fix $s\in\{0,\dots,2k-1\}$. We noted above that $\sum_{t=t_\ell}^{t_r-1}|d_t-\widetilde{d}_{t+\alpha}|\leq k$. It remains to show that $\sum_{t=t_\ell}^{t_r-1}|d_t-\widetilde{d}_{t+s}|> k$ if $s\neq \alpha$.
Let $s\in \{0,\dots,2k-1\}\setminus\{\alpha\}$.
We have
\begin{align}
\sum_{t=t_\ell}^{t_r-1}|d_t-\widetilde{d}_{t+s}|&=\sum_{t=t_\ell}^{t_r-1}|d_t-d_{t+s-\alpha}+d_{t+s-\alpha}-\widetilde{d}_{t+s}| \nonumber \\
&\geq \sum_{t=t_\ell}^{t_r-1}|d_t-d_{t+s-\alpha}|-\sum_{t=t_\ell}^{t_r-1}|d_{t+s-\alpha}-\widetilde{d}_{t+s}|. \label{eq:sum}
\end{align}
Since $\sum_{t=0}^{n-1}|d_t-\widetilde{d}_{t+\alpha}|\leq k$, it follows that
\begin{equation*} \sum_{t=t_\ell}^{t_r-1}|d_{t+s-\alpha}-\widetilde{d}_{t+s}|=\sum_{t=t_\ell+s-\alpha}^{t_r+s-\alpha-1}|d_{t}-\widetilde{d}_{t+\alpha}|\leq k.
\end{equation*}
Hence, \eqref{eq:sum} will be strictly greater than $k$ whenever $\sum_{t=t_\ell}^{t_r-1}|d_t-d_{t+s-\alpha}|>2k$.

Recall that the interval $[t_\ell,t_r-1]$ contains some interval $\mathcal{I}$ of $2k$ consecutive values of $t$ such that every $d_t$ is little. As $s\leq2k-1$ and $s\neq\alpha$, there exists some $\eta\in\mathbb{Z}$ such that $t_\ell+\eta(s-\alpha)\in\mathcal{I}$, where $\eta(s-\alpha)>0$. First assume $\eta>0$. Since $d_{t_\ell}$ is big and $d_{t_\ell+\eta(s-\alpha)}$ is little, we find
\begin{align*}
\sum_{t=t_\ell}^{t_r-1}|d_t-d_{t+s-\alpha}|&\geq\sum_{i=0}^{\eta-1}|d_{t_\ell+i(s-\alpha)}-d_{t_\ell+(i+1)(s-\alpha)}|\\
&\geq|d_{t_\ell}-d_{t_\ell+\eta(s-\alpha)}|\\
&\geq \sqrt{n}-\frac34\sqrt{n}=\frac14\sqrt{n}\\
&>2k.
\end{align*}
\noindent
If $\eta<0$, then $\alpha-s>0$ and
\begin{align*}
\sum_{t=t_\ell}^{t_r-1}|d_t-d_{t+s-\alpha}|&\geq\sum_{i=0}^{-\eta}|d_{t_\ell+(i+1)(\alpha-s)}-d_{t_\ell+i(\alpha-s)}| \geq|d_{t_\ell-\eta(\alpha-s)}-d_{t_\ell}|. 
\end{align*}
The result then follows in a similar fashion.
\end{claimproof}
\noindent
By Claim 4, we see that $\alpha$ is the only value $s\in\{0,\dots,2k-1\}$ satisfying $\sum_{t=t_\ell}^{t_r-1}|d_t-\widetilde{d}_{t+s}|\leq k$. As we have calculated $(d_t)_{t=t_\ell}^{t_r}$ and $(\widetilde{d}_t)$ from the cards, and we know $k$ as well, we are able to find the value  $s\in\{0,\dots,2k-1\}$ satisfying $\sum_{t=t_\ell}^{t_r-1}|d_t-\widetilde{d}_{t+s}|\leq k$, and hence identify $\alpha$. Once we have identified $\alpha$, we can then calculate $m=\widetilde{m}-\alpha$, the number of edges in $G$.
\end{proof}

\section{Conclusion}
\label{sec:conclusion}
We have shown that the size of a graph can be reconstructed if we are given a deck from which either at most $\missing\sqrt{n}$ cards are missing or at most $\adversary\sqrt{n}$ cards are false. The constants can be improved a little, although we do not know whether the result remains true with $\sqrt{n}$ missing cards. However, we suspect that stronger results could be proved by using more information about the degree sequences on the cards.

We also note that $c \sqrt{n}$ is still very far away from the best known lower bounds, which are linear. For example, for $n=3p+1$, Bowler, Brown and Fenner \cite{BBF10} have given the following two graphs which differ in the number of edges but have $\frac{2}3(n-1)$ cards in common: the graphs $G=2K_{p+1}+K_{p-1}$ and $H =K_{p+1}+2K_p$ both have $3p+1$ vertices and at least $2p$ cards of the form $K_{p+1}+K_p+K_{p-1}$.
We suspect that the lower bound is closer to the truth and propose the following question.
\begin{problem}
Does there exist some $\varepsilon>0$ such that, for any graph $G$ on $n$ vertices, we can reconstruct the number of edges of $G$ from any subset of at least $(1-\varepsilon)n$ cards?
\end{problem}
Another direction for future work is to reconstruct other graph parameters, such as the degree sequence or the number of triangles. Although our techniques do not immediately extend to this setting, we conjecture this should be possible from a partial deck as well.
\begin{conjecture}
\label{prob:counting}
Fix $k\in \mathbb{N}$ and a graph $H$ and let $n$ be sufficiently large. For every graph $G$ on $n$ vertices, the number of subgraphs of $G$ isomorphic to $H$ is reconstructible given any $n-k$ cards from $\mathcal{D}(G)$.
\end{conjecture}
\noindent
If we are given the entire deck $\mathcal{D}(G)$ (i.e. $k=0$), then this problem is solved by Kelly's Lemma \cite{K57}, which states that for any two graphs $G$ and $H$ with $|G|>|H|$, the number of subgraphs of $G$ isomorphic to $H$ is reconstructible.

If the number of edges is known, then the degree of a vertex can be calculated from the number of edges on its card. Therefore, by our main result, if $k\leq\missing\sqrt{n}$, then all but $k$ of the degrees are known. If $k$ is larger, then Lemma \ref{lem:findingdt} still allows us to construct most of the degree sequence. We expect that, for a large range of $k$, it is possible to determine the whole degree sequence exactly. As a first step, we make the following conjecture.
\begin{conjecture}
\label{conj:degree}
Fix $k\in\mathbb{N}$ and let $n$ be sufficiently large. For any graph $G$ on $n$ vertices, the degree sequence of $G$ is reconstructible from any $n-k$ cards.
\end{conjecture}
Note that a positive answer to Problem \ref{prob:counting} would give a positive answer to Conjecture \ref{conj:degree}: for fixed $k$ and $n$ sufficiently large, we can find the number of edges of the graph by Theorem \ref{thm:sizereconstructibility} and hence determine all but $k$ elements of the degree sequence. Provided $n$ is sufficiently large, we can reconstruct the number of copies of the star $K_{1,j}$ for $j=1,\dots,k+1$; this is given by $\sum_{v\in V(G)}\binom{d(v)}{j}$. By subtracting the terms corresponding to vertices of known degree, we obtain a sequence of polynomials in the unknown degrees. Adding constants, these form a basis for all polynomials of degree at most $k+1$. From these, it is straightforward to evaluate the remaining degrees.

\paragraph{Acknowledgements.} We would like to thank the referees for their helpful comments.

\end{document}